\documentclass[fleqn,10 pt]{article}
\usepackage{amsmath,amssymb,amsfonts,amsthm,graphicx}
\usepackage{cite}
\usepackage{epsfig}
\usepackage{epstopdf}
\usepackage{footnote}
\usepackage{mathtools}

\theoremstyle{plain}
\usepackage{booktabs}
\usepackage{hyperref}
\numberwithin{equation}{section}

\usepackage{hyperref}
\hypersetup{
	hidelinks=true,
	colorlinks   = true, 
	urlcolor     = blue, 
	linkcolor    = blue, 
	citecolor    = red 
}

\newtheorem{thm}{Theorem}[section]
\newtheorem{lm}[thm]{Lemma}

\theoremstyle{definition}

\newtheorem{example}{ Example}

\theoremstyle{remark}

\date{}

\title{A new approach of the Chebyshev wavelets  for the variable-order time  fractional mobile-immobile advection-dispersion model}
\author{M. H. Heydari\\
\footnotesize{Faculty of Mathematics, Yazd University, Yazd, Iran. }  \\
\footnotesize{The Laboratory of Quantum Information Processing, Yazd University, Yazd, Iran.}  \\
 \footnotesize{e-mail: heydari@stu.yazd.ac.ir.}}

\begin{document}
\maketitle

\begin{abstract}
This paper proposes a new numerical method based on the Chebyshev wavelets (CWs) to solve the variable-order time  fractional mobile-immobile advection-dispersion equation. To do this, a new  operational matrix of variable-order fractional derivative in the Caputo sense for the  CWs is derived and is used  to obtain an approximate solution for the problem under study. Along the way a new family of piecewise functions is introduced and employed to  derive a general method to compute this matrix.
The main advantage behind the proposed approach is that the problem under consideration is transformed into  a linear system of algebraic equations. So, it can be solved simply to obtain an approximate solution. The efficiency and accuracy  of the proposed method are shown for some concrete examples. These results show that the proposed method is  very  efficient and accurate.
\end{abstract}



\section{Introduction}
Variable-order fractional derivatives, which are an extension of constant-order fractional ones  have been introduced in several physical applications \cite{e1,e2,e3}. Recently, some  researchers  \cite{1,2,3,20,21,22,23,24} have shown that many complex physical models can be described via variable-order derivatives  with a great success.  It is worth noting that analytically handling equations described by the variable-order fractional derivatives is difficult due to their highly complex, so proposing efficient methods to find their numerical solutions is of great  importance in practical. So, recently several methods have been proposed to solve variable-order fractional differential equations numerically such as \cite{4,5,6,7,8,9,10,11,12,13,14,15,16,17,18,19}.\\
Wavelets theory which is a relatively new  area in mathematical research has been applied in a wide range of
engineering disciplines \cite{H0}. In recent years, wavelets have been applied for solving different types of partial differential equations e.g. \cite{H0,H1,H2,H3}.\\
The aim of this paper is to propose a new numerical method based on the CWs to solve the following  variable-order time fractional  mobile-immobile advection-dispersion model \cite{H01}:
\begin{equation}\label{D1}
\alpha_{1}\frac{\partial u(x,t)}{\partial t}+\alpha_{2}\,\prescript{c}{0}{D_{t}^{\gamma(x,t)}}u(x,t)=-\mu_{1}\frac{\partial u(x,t)}{\partial x}+\mu_{2}\frac{\partial^{2} u(x,t)}{\partial x^{2}}+f(x,t),\hspace{0.5cm}(x,t)\in \Omega,
\end{equation}
with $\Omega=[0,1]\times [0,1]$, subject to the  following initial-boundary conditions:
\begin{equation}\label{D2}
\begin{array}{ll}
  u(x,0)=g(x), &  \\  \noalign{\medskip}
  u(0,t)=h_{1}(t), & u(1,t)=h_{2}(t),
\end{array}
\end{equation}
where $\alpha_{1}\geq 0$, $\alpha_{2}\geq 0$, $\mu_{1}> 0$, $\mu_{2}> 0$,  $0<\gamma(x,t)\leq 1$, $f(x,t)$, $g(x)$, $h_{1}(t)$ and $h_{2}(t)$ are given functions, and   $\prescript{c}{0}{D_{t}^{\gamma(x,t)}}$ denotes the variable-order fractional derivatives in the Caputo sense of order $0<\gamma(x,t)\leq 1$, as defined by \cite{11,12}:
\begin{equation}\label{de1}
 \prescript{c}{0}{D_{t}^{\gamma(x,t)}}u(x,t)=\frac{1}{\Gamma\left(1-\gamma(x,t)\right)}\int_{0}^{t}\left(t-s\right)^{-\gamma(x,t)}\frac{\partial u(x,s)}{\partial s}ds,\hspace{0.5cm}t>0.
\end{equation}
It is worth noting that based on the definition of the variable-order fractional derivative in the Caputo sense as \cite{12}, we have the following useful property:
\begin{equation}\label{de3}
\prescript{c}{0}{D_{t}^{\vartheta(x,t)}}t^{m}=
\displaystyle \left\{\begin{array}{cl}
\displaystyle\frac{\Gamma(m+1)}{\Gamma(m-\vartheta(x,t)+1)}\,t^{m-\vartheta(x,t)}, & q\leq m \in \mathbb{N}, \\ \noalign{\medskip}
0, & o.w,
\end{array}
\right.
\end{equation}
where $q-1<\vartheta (x,t)\leq q$.\\
 To solve equation in (\ref{D1}), we first derive a new  operational matrix of variable-order fractional derivative in the Caputo sense for the  CWs and then, employ this matrix to obtain an approximate solution for the problem at hand. Along the way, a new family of piecewise functions is introduced and employed to  derive a general procedure for forming this matrix.\\
In the proposed method, at first the solution of the problem at hand is expanded in terms of  the  CWs. Then, by computing
the operational matrix of variable-order fractional derivative and using some properties of these basis polynomials, we  transform
 its solution to the solution of a linear system of algebraic equations.  This greatly simplifies the process of solving the problem as well as help to achieve  an approximate solution for the problem.\\
The  remainder of this paper is organized as follows: In section \ref{S2},  the   CWs and their properties  are introduced. In section \ref{S3}, the operational matrix of variable-order fractional derivative  for the  CWs is derived and in section \ref{S4},
 the proposed  method is described for solving the problem under study.   Section \ref{S5}, contains some numerical examples which are solved using the proposed method. Finally, a conclusion is given in section \ref{S6}.
 \section{The CWs  and their properties}\label{S2}
 Wavelets constitute a family of functions constructed from dilation and translation of a single function $\psi(t)$ called the
mother wavelet. When the dilation parameter $a$ and the translation parameter $b$ vary continuously we have the following family of continuous wavelets as:
\begin{equation}
\psi_{ab}(t)=|a|^{-\frac{1}{2}}\psi\left(\frac{t-b}{a}\right),\hspace{0.5cm}a,\,b\in\mathbb{R},\,\,a\neq0.
\end{equation}
If we restrict the parameters $a$ and $b$ to discrete values as $a=a_{0}^{-k}$, $b=nb_{0}a_{0}^{-k}$, where $a_{0}>1$, $b_{0}>0$, we have the following family of discrete wavelets
\begin{equation}
\psi_{kn}(t)=|a_{0}|^{\frac{k}{2}}\psi\left(a_{0}^{k}t-nb_{0}\right),\hspace{0.5cm}k,\,n\in\mathbb{Z},
\end{equation}
where the functions $\psi_{kn}(t)$ form a wavelet basis for $L^{2}(\mathbb{R})$.\\
 In practice, when $a_{0}=2$ and $b_{0}=1$, the functions $\psi_{kn}(t)$ form an orthonormal basis.\\
The CWs  are defined on the interval $[0,1]$ by \cite{MM1,MM2}:
\begin{equation}
\psi_{nm}(t)=\displaystyle\left\{\begin{array}{ccl}
\displaystyle \beta_{m}2^{\frac{k}{2}}T^{\ast}_{m}\left(2^{k}t-n\right), & &\displaystyle t\in\left[\frac{n}{2^{k}},\frac{n+1}{2^{k}}\right], \\ \noalign{\medskip}
0, && o.w,
\end{array}\right.
\end{equation}
for $n=0,1,\ldots,2^{k}-1$, $m=0,1,\ldots,M-1$, $(k,M)\in \mathbb{N}$, where
\begin{equation}
\beta_{m}=\displaystyle \left\{\begin{array}{ccc}
\displaystyle \sqrt{\frac{2}{\pi}}, && m=0, \\  \noalign{\medskip}
 \displaystyle \frac{2}{\sqrt{\pi}}, && m\geq 1,
 \end{array}\right.
\end{equation}
and $T^{\ast}_{m}(t)$ denotes the shifted Chebyshev polynomials, which are defined on the interval $[0,1]$ as:
\begin{equation}
\begin{array}{l}
 \displaystyle T^{\ast}_{0}(t)=1, \\  \noalign{\medskip}
 \displaystyle T^{\ast}_{m}(t)=m\sum_{i=0}^{m}(-1)^{m-i}\frac{2^{2i}(m+i-1)!}{(m-i)!(2i)!}\,t^{i},\hspace{0.5cm}m=1,2,\ldots.
\end{array}
\end{equation}
The set of the CWs is an orthogonal set with respect to the weight function $w_{n}(t)$ where
\begin{equation}
\displaystyle w_{n}(t)=\left\{\begin{array}{ccl}
\displaystyle \frac{1}{\sqrt{1-\left(2^{k+1}t-2n-1\right)^2}}, && \displaystyle t\in\left[\frac{n}{2^{k}},\frac{n+1}{2^{k}}\right], \\ \noalign{\medskip}
0, && o.w.
\end{array}
\right.
\end{equation}
The CWs can be used to expand any function $u(t)$ which is defined over $[0,1]$ as:
\begin{equation}\label{n0}
u(t)=\sum_{n=0}^{\infty}\sum_{m=0}^{\infty}{c_{nm}\psi_{nm}(t)},
\end{equation}
where $c_{nm}=\left< u(t),\psi_{nm}(t)\right>$ and $\left<.,.\right>$ denotes the inner product in $L_{w_{n}}^{2}[0,1]$.\\
 By truncating  the infinite series in equation (\ref{n0}), $u(t)$ is approximated as:
\begin{equation}
\label{e27} u(t)\simeq
\sum_{n=0}^{2^{k}-1}\sum_{m=0}^{M-1}{c_{nm}\psi_{nm}(t)=C^{T}\Psi(t)},
\end{equation}
where  $C$ and $\Psi(t)$ are  column vectors with $\hat{m}=2^{k}M$ elements.\\
For simplicity, equation (\ref{e27}) is written as:
\begin{equation}
 u(t)\simeq
\sum_{i=1}^{\hat{m}}{c_{i}\psi_{i}(t)=C^{T}\Psi(t)},
\end{equation}
where $c_{i}=c_{nm}$ and $\psi_{i}(t)=\psi_{nm}(t)$, and the  index $i$ is calculated as $i=Mn+m+1$.\\ Thus, we have:
\begin{equation} \label{e28}
C\triangleq\left[c_{1},c_{2},\ldots,c_{\hat{m}}\right]^{T},\nonumber
\end{equation}
\begin{equation}
\label{e9}
\Psi(t)\triangleq\left[\psi_{1}(t),\psi_{2}(t),\ldots,\psi_{\hat{m}}(t)\right]^{T}.
\end{equation}
Similarly, the CWs can be used to expand an arbitrary function of two variables such as $u(x,t)$ which is defined over $[0,1]\times
[0,1]$ as:
\begin{equation}\label{TT1}
u(x,t)\simeq\sum_{i=1}^{\hat{m}}\sum_{j=1}^{\hat{m}}
u_{ij}\psi_i(x)\psi_j(t)=\Psi^T(x)\textbf{U}\Psi(t),
\end{equation}
 where $U=[u_{ij}]$ and
 $u_{ij}=\left(\psi_{i}(x),\left(u(x,t),\psi_{j}(t)\right)\right)$.\\
The derivative of the vector $\Psi(t)$ defined in equation (\ref{e9}) can be expressed as \cite{MM3}:
\begin{equation}\label{C8}
\frac{d\Psi(t)}{dt}=\textbf{D}\Psi(t),
\end{equation}
where $\textbf{D}$ is the $\hat{m}\times \hat{m}$ operational matrix of one-time derivative of the CWs vector and is given by:
\begin{equation}
\textbf{D}=\left( \begin{array}{cccc}
      \textbf{F} & 0 & \ldots & 0 \\
      0 & \textbf{F} & \ldots & 0 \\
      \vdots & \vdots & \ddots & \vdots \\
      0 & 0 & 0 & \textbf{F} \\
    \end{array}
  \right),
\end{equation}
where $\textbf{F}$ is an $M\times M$ matrix with the elements:
\begin{equation}
\textbf{F}_{ij}=\displaystyle \left\{\begin{array}{cl}
\displaystyle 2^{k+2}(i-1)\sqrt{\frac{\sigma_{i-1}}{\sigma_{j-1}}}, & i=2,3,\ldots,M,\,\,\, j=1,2,\ldots,i-1,\,\,\,i+j~ \text{is odd},
\\ \noalign{\medskip}
  0, & o.w,
 \end{array} \right.
\end{equation}
and
\begin{equation}
\displaystyle \sigma_{j}=\left\{\begin{array}{cc}
 2, & j=0, \\ \noalign{\medskip}
 1, & j\geq 1.
 \end{array}\right.\nonumber
\end{equation}
 In general, the operational matrix $\textbf{D}^{r}$ of $r$-times derivative of $\Psi(t)$ can be expressed as:
\begin{equation}\label{C9}
 \frac{d^{r}\Psi(t)}{dt^{r}}=\textbf{D}^{r}\Psi(t),
\end{equation}
where $\textbf{D}^{r}$ is the $r$-th power of matrix $\textbf{D}$.
\section{The operational matrix of variable-order fractional derivative }\label{S3}
The  variable-order fractional derivative of order $(q-1)< \vartheta (x,t)\leq q,\,q\in \mathbb{N}$, of the vector $\Psi(t)$ which is defined in equation (\ref{e9}) can be expressed as:
\begin{equation}\label{O1}
\prescript{c}{0}{D_{t}^{\vartheta(x,t)}}\Psi(t)\simeq \textbf{Q}_{t}^{\vartheta(x,t)}\Psi(t),
\end{equation}
where $Q^{\vartheta(x,t)}$ is called the $\hat{m}\times\hat{m}$ operational matrix of variable-order fractional derivative of order $\vartheta(x,t)$ for the CWs.\\
 In the sequel, we will derive an  explicit form for this matrix. To this end, we  introduce another family of piecewise functions, which are defined on $[0,1]$ as:
\begin{equation}\label{phi}
\phi_{nm}(t)=\displaystyle\left\{\begin{array}{ccc}
\displaystyle t^{m}, & &\displaystyle t\in\left[\frac{n}{2^{k}},\frac{n+1}{2^{k}}\right], \\ \noalign{\medskip}
0, && o.w,
\end{array}\right.
\end{equation}
for $n=0,1,\ldots,2^{k}-1$, $m=0,1,\ldots,M-1$.\\
Unlike the CWs, this family of functions is not normalized. An $\hat{m}$-set of these  functions may be expressed as:
\begin{equation}
\Phi(t)\triangleq [\phi_{1}(t),\phi_{2}(t),\ldots,\phi_{\hat{m}}(t)]^{T},
\end{equation}
where $\phi_{i}(t)=\phi_{nm}(t)$, and the  index $i$ is determined by the relation $i=Mn+m+1$.\\
The following relation  holds among these functions and the CWs:
\begin{equation}\label{co}
\Phi(t)=\textbf{P}\Psi(t),
\end{equation}
where $p_{ij}=(\phi_{i},\psi_{j})$.
\begin{lm}\label{lem1}
Let $\phi_{nm}(t)$ be as defined in equation (\ref{phi}), and $(q-1)<\vartheta(x,t)\leq q$ be a positive function defined over $[0,1]$. Then, we have:
\begin{equation}
\prescript{c}{0}{D_{t}^{\vartheta(x,t)}}\phi_{nm}(t)=\displaystyle\left\{\begin{array}{ccl}
\displaystyle \frac{m!}{\Gamma(m-\vartheta(x,t)+1)}\, t^{m-\vartheta(x,t)}, & m=q,q+1,\ldots,M-1,
              &\displaystyle t\in\left[\frac{n}{2^{k}},\frac{n+1}{2^{k}}\right], \\  \noalign{\medskip}
0, && o.w.
\end{array}\right.\nonumber
\end{equation}
\end{lm}
\begin{proof}
By considering  relation (\ref{de3}), the proof will be  straightforward.
\end{proof}
\begin{thm}\label{th1}
Let $\Phi(t)$ be the  piecewise functions vector  defined as in equation (\ref{phi}) and $(q-1)<\vartheta(x,t)\leq q$ be a positive function defined over $[0,1]$. The variable-order fractional derivative of order $\vartheta(x,t)$ in the Caputo sense of  $\Phi(t)$ can be expressed by:
\begin{equation}
\prescript{c}{0}{D_{t}^{\vartheta(x,t)}}\Phi(t)=\textbf{T}_{t}^{\vartheta(x,t)}\Phi(t),\nonumber
\end{equation}
where $\textbf{T}_{t}^{\vartheta(x,t)}$ is an $\hat{m}\times \hat{m}$ matrix given by:
\begin{equation}\label{pe}
\textbf{T}_{t}^{\vartheta(x,t)}=\left(
\begin{array}{ccccc}
\displaystyle \textbf{G}_{t}^{\vartheta(x,t)} & 0 & 0 & \ldots & 0 \\ \noalign{\medskip}
0 & \displaystyle \textbf{G}_{t}^{\vartheta(x,t)} & 0 & \ldots & 0 \\ \noalign{\medskip}
 0 & 0 & \textbf{G}_{t}^{\vartheta(x,t)} & \ldots & 0 \\ \noalign{\medskip}
  \vdots & \vdots & \vdots & \ddots & \vdots \\ \noalign{\medskip}
   0 & 0 & 0 & 0 & \displaystyle \textbf{G}_{t}^{\vartheta(x,t)} \\
    \end{array}
     \right),
\end{equation}
and $\textbf{G}_{t}^{\vartheta(x,t)}$ is an $M\times M$ matrix given as:
\begin{equation}
\resizebox{.85\textwidth}{!}{\[
\textbf{G}_{t}^{\vartheta(x,t)}=t^{-\vartheta(x,t)}\left(
\begin{array}{cccccccc}
0 &0&  \ldots & 0 & 0 & 0 & \ldots & 0 \\
\vdots &\vdots&  \vdots & \vdots & \vdots & \vdots & \vdots & \vdots\\
 0 &0&  0 & 0 & 0 & 0 & 0 & 0\\
 0 &\ldots&  0 & \displaystyle \frac{q!}{\Gamma(q-\vartheta(x,t)+1)}& 0 & 0 & \ldots & 0 \\ \noalign{\medskip}
 0 & \ldots& \ldots & 0 & \displaystyle \frac{(q+1)!}{\Gamma(q-\vartheta(x,t)+2)} & 0 & 0 & 0 \\ \noalign{\medskip}
 \vdots& \vdots & \vdots & \vdots & \vdots & \ddots &0 & 0 \\ \noalign{\medskip}
 \vdots & \vdots& \vdots & \vdots & \vdots & \vdots &\displaystyle \frac{(M-2)!}{\Gamma(M-\vartheta(x,t)-1)} & 0 \\ \noalign{\medskip}
  0 & 0&\ldots & 0 & 0 & 0 & 0 &\displaystyle \frac{(M-1)!}{\Gamma(M-\vartheta(x,t))}
  \end{array}
   \right). \]}\nonumber
\end{equation}
\end{thm}

\begin{proof}
It is an immediate consequence of Lemma \ref{lem1}.
\end{proof}
\noindent To illustrate the calculation procedure, we choose $(M=5,\,k=1)$ and $1<\vartheta(x,t)\leq 2$. Thus, we have:
\begin{equation}
\textbf{G}_{t}^{\vartheta(x,t)}=t^{-\vartheta(x,t)}\left(
\begin{array}{ccccc}
 0 & 0 & 0 & 0 & 0 \\
 0 & 0 & 0 & 0 & 0 \\
 0 & 0 & \displaystyle \frac{2}{\Gamma(3-\vartheta(x,t))} & 0 & 0 \\ \noalign{\medskip}
  0 & 0 & 0 & \displaystyle \frac{6}{\Gamma(4-\vartheta(x,t))} & 0 \\ \noalign{\medskip}
    0 & 0 & 0 & 0 & \displaystyle \frac{24}{\Gamma(5-\vartheta(x,t))} \\
    \end{array}
     \right).\nonumber
\end{equation}
\begin{thm}
Let $\Psi(t)$ be the  CWs vector  defined in equation (\ref{e9}) and $(q-1)<\vartheta(x,t)\leq q$ be a positive function defined over $[0,1]$. The variable-order fractional derivative of order $\vartheta(x,t)$ in the Caputo sense of  $\Psi(t)$ can be expressed as:
\begin{equation}\label{o2}
\prescript{c}{0}{D_{t}^{\vartheta(x,t)}}\Psi(t)=\textbf{Q}_{t}^{\vartheta(x,t)}\Psi(t)=\left(\textbf{P}^{-1}\textbf{T}_{t}^{\vartheta(x,t)}\textbf{P}\right)\Psi(t),
\end{equation}
where $\textbf{P}$ is the  coefficients matrix defined in equation (\ref{co}),  $\textbf{T}_{t}^{\vartheta(x,t)}$ is the operational matrix of variable-order fractional derivative of order $\vartheta(x,t)$ for the piecewise functions, which is defined in equation (\ref{pe}) and $\textbf{Q}_{t}^{\vartheta(x,t)}$ is called the operational matrix of variable-order fractional derivative of order $\vartheta(x,t)$ for the CWs.
\end{thm}
\begin{proof}
By considering equation (\ref{co}) and Theorem \ref{th1}, we have:
\begin{equation}
\Psi(t)=\textbf{P}^{-1}\Phi(t),\nonumber
\end{equation}
and then
\begin{equation}
\prescript{c}{0}{D_{t}^{\vartheta(x,t)}}\Psi(t)=\textbf{P}^{-1}\prescript{c}{0}{D_{t}^{\vartheta(x,t)}}\Phi(t)=
\textbf{P}^{-1}\textbf{T}_{t}^{\vartheta(x,t)}\Phi(t)=\left(\textbf{P}^{-1}\textbf{T}_{t}^{\vartheta(x,t)}\textbf{P}\right)\Psi(t),\nonumber
\end{equation}
which completes the proof.
\end{proof}
\noindent To illustrate the calculation procedure, we choose $(M=3,\,k=1)$ and $0<\vartheta(x,t)\leq 1$. Thus, we have:
\begin{equation}
\textbf{Q}_{t}^{\vartheta(x,t)}=t^{-\vartheta(x,t)}
\left(
  \begin{array}{cc}
    \textbf{A} & 0  \\
    0&\textbf{B}
  \end{array}
 \right), \nonumber
\end{equation}
where
\begin{equation}
\textbf{A}=\left(
  \begin{array}{ccc}
    0 & 0 & 0 \\
  \displaystyle \frac{\sqrt{2}}{\Gamma(2-\vartheta(x,t))} &\displaystyle \frac{1}{\Gamma(2-\vartheta(x,t))}& 0\\ \noalign{\medskip}
   \displaystyle  \frac{-4\sqrt{2}}{\Gamma(2-\vartheta(x,t))}+\frac{6\sqrt{2}}{\Gamma(3-\vartheta(x,t))} & \displaystyle \frac{-4}{\Gamma(2-\vartheta(x,t))}+\frac{8}{\Gamma(3-\vartheta(x,t))} & \displaystyle \frac{2}{\Gamma(3-\vartheta(x,t))}
   \end{array}
 \right),\nonumber
\end{equation}
and
\begin{equation}
\textbf{B}=\left(
  \begin{array}{ccc}
    0 & 0 & 0 \\
     \displaystyle  \displaystyle \frac{3\sqrt{2}}{\Gamma(2-\vartheta(x,t))} & \displaystyle \frac{1}{\Gamma(2-\vartheta(x,t))}  & 0 \\ \noalign{\medskip}
    \displaystyle\displaystyle \frac{-36\sqrt{2}}{\Gamma(2-\vartheta(x,t))}+\frac{38\sqrt{2}}{\Gamma(3-\vartheta(x,t))} & \displaystyle \frac{-12}{\Gamma(2-\vartheta(x,t))}+\frac{24}{\Gamma(3-\vartheta(x,t))} & \displaystyle\frac{2}{\Gamma(3-\vartheta(x,t))}
    \end{array}
 \right).\nonumber
\end{equation}

\section{Description of the proposed method}\label{S4}
In this section,  the CWs expansion and their operational matrix of variable-order fractional derivative   are used together  to solve  the  variable-order time fractional  mobile-immobile advection-dispersion model of equation (\ref{D1}). To solve this equation,  we approximate the unknown function  $u(x,t)$ by the  CWs as:
\begin{equation}
\label{D4}
u(x,t) \simeq \Psi(x)^{T}\textbf{U}\Psi(t),
\end{equation}
where $\textbf{U}=[u_{ij}]$ is an $\hat{m}\times \hat{m}$ matrix which we need to compute it, and $\Psi(.)$ is the CWs vector, which is defined in equation (\ref{e9}).\\
By derivatives of equation (\ref{D4}) for one time with respect to $t$ and two times with respect to $x$,  and using equations (\ref{C8}) and (\ref{C9}), we obtain:
\begin{equation}\label{D5}
\begin{array}{lll}
  u_{t}(x,t)\simeq  \Psi(x)^{T}\textbf{U}\textbf{D}\Psi(t), & u_{x}(x,t)\simeq  \Psi(x)^{T}\textbf{D}^{T}\textbf{U}\Psi(t), & u_{xx}(x,t)\simeq  \Psi(x)^{T}(\textbf{D}^{2})^{T}\textbf{U}\Psi(t).
\end{array}
\end{equation}
By the variable-order fractional derivative of order $\gamma(x,t)$ of equation (\ref{D4})  with respect to $t$,  and considering equation (\ref{o2}), we have:
\begin{equation}\label{D6}
 \prescript{c}{0}{D_{t}^{\gamma(x,t)}}u(x,t)\simeq  \Psi(x)^{T}\textbf{U}\textbf{Q}^{\gamma(x,t)}\Psi(t).
\end{equation}
Applying equation (\ref{D4})  into the initial-boundary conditions expressed in equation (\ref{D2}), and using equation (\ref{C8}), we have:
\begin{equation}\label{D9}
\begin{array}{lll}
 \Lambda_{1}(x)=\Psi(x)^{T}\textbf{U}\Psi(0)-g(x)\simeq0,& \Lambda_{2}(t)=\Psi(0)^{T}\textbf{U}\Psi(t)-h_{1}(t)\simeq0, &\Lambda_{1}(t)=\Psi(1)^{T}\textbf{U}\Psi(t)-h_{2}(t)\simeq0.
  \end{array}
\end{equation}
By substituting equations (\ref{D5}) and (\ref{D6}) into the variable-order time  fractional mobile-immobile advection-dispersion model in equation (\ref{D1}), we get:
\begin{equation}\label{D8}
 \Psi(x)^{T}\left(\alpha_{1}\textbf{U}\textbf{D}+\alpha_{2}\textbf{U}\textbf{Q}^{\alpha(x,t)}+\mu_{1}\textbf{D}^{T}\textbf{U}-\mu_{2}(\textbf{D}^{2})^{T}\textbf{U}\right)\Psi(t)-f(x,t)\triangleq F(x,t)\simeq 0.
\end{equation}
In order to obtain an approximate solution for the  problem at hand, we need to find the unknown matrix $\textbf{U}$.
So, we need to construct a linear system of $\hat{m}^{2}$ equations which by solving it, the unknown matrix $\textbf{U}$ is determined. To this end, we choose $\hat{m}^{2}-3\hat{m}+2$ algebraic equations using equation (\ref{D8}) as:
\begin{equation}\label{T11}
F(x_{i},t_{j})=0,\hspace{0.5cm}i=2,3,\ldots,\hat{m}-1,\,\,j=2,3,\ldots,\hat{m},
\end{equation}
where $x_{i}$ and $t_{j}$ are the zeros of the shifted  Chebyshev polynomials of degree $\hat{m}$ on [0,1].\\
Moreover, by taking the collocation points $x_{i}$ and $t_{i}$ into equation (\ref{D9}) as:
\begin{equation}\label{T12}
\begin{array}{ll}
\Lambda_{1}(x_{i})=0,  &i=1,2,\ldots,\hat{m}, \\ \noalign{\medskip}
\Lambda_{2}(t_{i})=0,  &i=2,3,\ldots,\hat{m},\\ \noalign{\medskip}
\Lambda_{3}(t_{i})=0,  &i=2,3,\ldots,\hat{m},\\ \noalign{\medskip}
\end{array}
\end{equation}
we get $3\hat{m}-2$ linear algebraic equations.\\
Combining equations (\ref{T11}) and (\ref{T12}) gives a linear system of $\hat{m}^{2}$ algebraic equations, which can be solved for the unknown matrix $\textbf{U}=[u_{ij}],\,i,j=1,2,\ldots,\hat{m}$, using MAPLE or MATLAB software packages. By determining $\textbf{U}$, we can determine the approximate solutions for $u(x,t)$  from equation (\ref{D4}).
\section{Illustrative test problems}\label{S5}
In this section, we provide some numerical examples to demonstrate  the efficiency and reliability  of our method.
It is worth mentioning that all the numeric computations are performed by MAPLE 15  with 50 decimal digits.
\begin{example}\label{ex1}
Consider the variable-order time fractional  problem \cite{H01}:
\begin{equation}
\alpha_{1}\frac{\partial u(x,t)}{\partial t}+\alpha_{2}\,\prescript{c}{0}{D_{t}^{\gamma(x,t)}}u(x,t)=-\mu_{1}\frac{\partial u(x,t)}{\partial x}+\mu_{2}\frac{\partial^{2} u(x,t)}{\partial x^{2}}+f(x,t),\nonumber
\end{equation}
 subject to the  following initial-boundary conditions:
\begin{equation}
\begin{array}{lll}
  u(x,0)=10x^{2}(1-x)^{2}, & u(0,t)=0, & u(1,t)=0,
  \end{array}\nonumber
\end{equation}
where
\begin{equation}
f(x,t)=10 \,\left(\alpha_{1}+\alpha_{2}\frac{t^{1-\gamma(x,t)}}{\Gamma(2-\gamma(x,t))}\right)x^{2}\left(1-x\right)^{2}+10\,\left(\mu_{1} \left(4 x^{3}-6 x^{2}+2 x\right)-\mu_{2}\left(12 x^{2}-12x+2 \right)\right)(t+1).\nonumber
\end{equation}
The analytical solution for this problem is $u(x,t)=10\,(t+1)x^{2}(1-x)^{2}$.
The numerical solution for this problem is also computed by our method for $\displaystyle \gamma(x,t)=1-0.5e^{-(xt)}$, $\alpha_{1}=\alpha_{2}=\mu_{1}=\mu_{2}=1$  and $(k=1,M=5)$. The numerical behavior of the approximate solution  and absolute error  are shown in Fig. \ref{Fg1}. From Fig. \ref{Fg1} it can be seen that the proposed method  is very efficient and accurate for solution of this problem. It is also worth noting that in \cite{H1}, the authors have proposed a discrete implicit numerical method for solving this problem. By considering Fig. \ref{Fg1} and Tables 1 and 2 in \cite{H1}, one can simply see that  the approximate solution obtained by  the method of this paper is more accurate the one in \cite{H1}. Moreover, the implementation of  our proposed method is much  simple in comparison with the  one in \cite{H1}.
\end{example}
\begin{example}\label{ex2}
Consider the following variable-order time fractional  problem:
\begin{equation}
\alpha_{1}\frac{\partial u(x,t)}{\partial t}+\alpha_{2}\,\prescript{c}{0}{D_{t}^{\gamma(x,t)}}u(x,t)=-\mu_{1}\frac{\partial u(x,t)}{\partial x}+\mu_{2}\frac{\partial^{2} u(x,t)}{\partial x^{2}}+f(x,t),\nonumber
\end{equation}
with  homogenous initial-boundary conditions and
\begin{equation}
f(x,t)=10\,\left(\alpha_{1}\left(3t^{2}-4t^{3}\right)+\alpha_{2}\left(\frac{6\,t^{3-\gamma(x,t)}}{\Gamma(4-\gamma(x,t))}-\frac{24\,t^{4-\gamma(x,t)}}{\Gamma(5-\gamma(x,t))}\right)\right)x^{3}\left(1-x\right)\nonumber
\end{equation}
\begin{equation}
\hspace{1.2cm}+10\,\left(\mu_{1} \left(3x^{2}-4x^3\right)-\mu_{2}\left(6x-12x^{2}\right)\right)t^{3}(1-t).\nonumber
\end{equation}
Its analytical solution is $u(x,t)=10\,x^{3}t^{3}(1-x)(1-t)$.
It is also solved numerically by our method for $\displaystyle \gamma(x,t)=1-0.4\sin(x+t)^{2}$, $\alpha_{1}=\alpha_{2}=1,\,\mu_{1}=\mu_{2}=2$  and $(k=1,M=5)$. The numerical behavior of the approximate solution  and absolute error  are shown in Fig. \ref{Fg2}. From Fig. \ref{Fg2}, it can be seen that our method  is very efficient and accurate for solving this problem.
\end{example}
\begin{example}\label{ex3}
Consider the following variable-order time fractional  problem:
\begin{equation}
\alpha_{1}\frac{\partial u(x,t)}{\partial t}+\alpha_{2}\,\prescript{c}{0}{D_{t}^{\gamma(x,t)}}u(x,t)=-\mu_{1}\frac{\partial u(x,t)}{\partial x}+\mu_{2}\frac{\partial^{2} u(x,t)}{\partial x^{2}}+f(x,t),\nonumber
\end{equation}
 subject to the  following initial-boundary conditions:
\begin{equation}
\begin{array}{lll}
  u(x,0)=0, &   u(0,t)=t^{3}, & u(1,t)=e\,t^{3},
 \end{array}\nonumber
\end{equation}
where
\begin{equation}
f(x,t)=\left(\alpha_{1}\,3t^{2}+\alpha_{2}\,\frac{6t^{3-\gamma(x,t)}}{\Gamma(4-\gamma(x,t))}+\left(\mu_{1}-\mu_{2}\right)t^{3}\right)x\left(1-x\right)e^{x}.\nonumber
\end{equation}
The analytical solution for this problem is $u(x,t)=t^{3}e^{x}$.
It is also solved numerically by our method for $\displaystyle \gamma(x,t)=0.8+0.2e^{-x}\sin(t)$, $\alpha_{1}=1,\,\alpha_{2}=\frac{1}{2},\,\mu_{1}=1,\,\mu_{2}=2$, $k=0$,  and  some different values of $M$.
The  absolute errors of the approximate solution at  $x=0.5$ for some different values of $t$ are shown in Table \ref{T1}.
The numerical behavior of the approximate solution and absolute error for $M=11$ are shown in Fig. \ref{Fg3}.
From Table \ref{T1}, we observe that the proposed method can provide  numerical results with high accuracy in all cases.
  Furthermore, it can be seen that the accuracy of the obtained results is improved by increasing the number of the CWs. From Fig. \ref{Fg3}, it can be seen that our method is very efficient and accurate for solution of this problem.
\end{example}
\begin{example}\label{ex4}
Consider the following variable-order time fractional  problem:
\begin{equation}
\frac{\partial u(x,t)}{\partial t}+\,\prescript{c}{0}{D_{t}^{\gamma(x,t)}}u(x,t)=-\frac{\partial u(x,t)}{\partial x}+\frac{\partial^{2} u(x,t)}{\partial x^{2}}+f(x,t),\nonumber
\end{equation}
 subject to the  following initial-boundary conditions:
\begin{equation}
\begin{array}{lll}
  u(x,0)=0, &  u(0,t)=t^{3}, & u(1,t)=t^{3},
 \end{array}\nonumber
\end{equation}
where
\begin{equation}
f(x,t)=\left( 3\,{t}^{2}+6\,{\frac {{t}^{3-\alpha \left( x,t \right) }}{
\Gamma  \left( 4-\alpha \left( x,t \right)  \right) }} \right)
 \left(  \left| 2\,x-1 \right|  \right) ^{3}+ \left( 6\, \left| 2\,x-1
 \right|  \left( 2\,x-1 \right) -24\, \left| 2\,x-1 \right|  \right) {
t}^{3}
.\nonumber
\end{equation}
Its analytical solution is $u(x,t)=t^{3}\left(|2x-1|\right)^{3}$.
It is also solved numerically by the our method for $\displaystyle \gamma(x,t)=1-e^{-xt}$  and $(k=1,M=4)$. The numerical behavior of the approximate solution  and absolute error  are shown in Fig. \ref{Fg4}. By Fig. \ref{Fg4}, it can be seen that our method  is very efficient and accurate for solving this problem.
\end{example}
\begin{table}[h!]
\caption{ \small{The  absolute errors of the approximate solution at  $x=0.5$ for different values of $t$ for Example \ref{ex3}.}}
 \centering
 \begin{tabular}{ccccccccc}
  \toprule
   \noalign{\medskip}
    $t$ & $M=4$ & $M=5$ & $M=6$ & $M=7$ & $M=8$&$M=9$&$M=10$&$M=11$ \\
     \toprule
  0.1 & 7.596E-07 & 1.492E-09 & 1.785E-09& 2.336E-12 & 6.160E-13 &5.243E-16&1.089E-15&9.735E-19\\
  0.2 & 9.642E-08 & 1.492E-09 & 1.358E-08& 2.325E-11 & 2.099E-12 &1.240E-14&1.114E-14&1.336E-17\\
  0.3 & 4.509E-07 & 2.284E-08 & 5.224E-08& 9.903E-11 & 8.307E-13 &6.190E-14&4.314E-14&5.765E-17\\
  0.4 & 5.420E-06 & 6.907E-08 & 1.350E-07& 2.687E-10 & 1.491E-11 &1.798E-13&1.114E-13&1.574E-16\\
  0.5 & 1.783E-05 & 1.571E-07 & 2.796E-07& 5.732E-10 & 4.758E-11 &3.989E-13&2.306E-13&3.372E-16\\
  0.6 & 4.070E-05 & 3.010E-07 & 5.041E-07& 1.054E-09 & 1.064E-10 &7.525E-13&4.157E-13&6.219E-16\\
  0.7 & 7.706E-05 & 5.150E-07 & 8.264E-07& 1.754E-09 & 1.994E-10 &1.274E-12&6.816E-13&1.036E-15\\
  0.8 & 1.299E-04 & 8.138E-07 & 1.265E-06& 2.717E-09 & 3.342E-10 &1.998E-12&1.043E-12&1.607E-15\\
  0.9 & 2.022E-04 & 1.212E-06 & 1.838E-06& 3.984E-09 & 5.188E-10 &2.959E-12&1.515E-12&2.359E-15\\
   \toprule
      \end{tabular}
 \label{T1}
\end{table}
\begin{figure}[h!]
\begin{tabular}{ll}
\epsfig{file=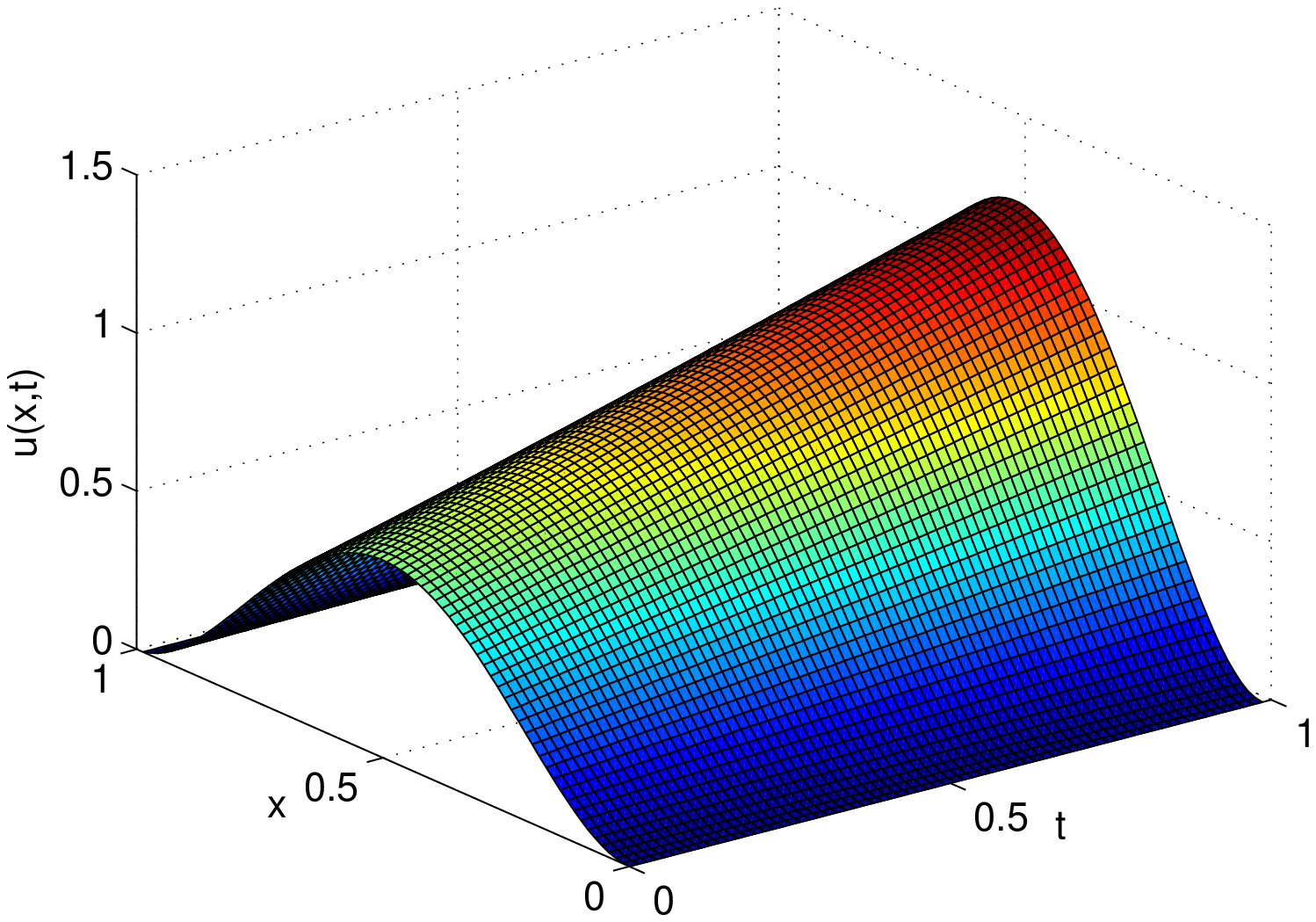,scale=0.48,clip=} &
\epsfig{file=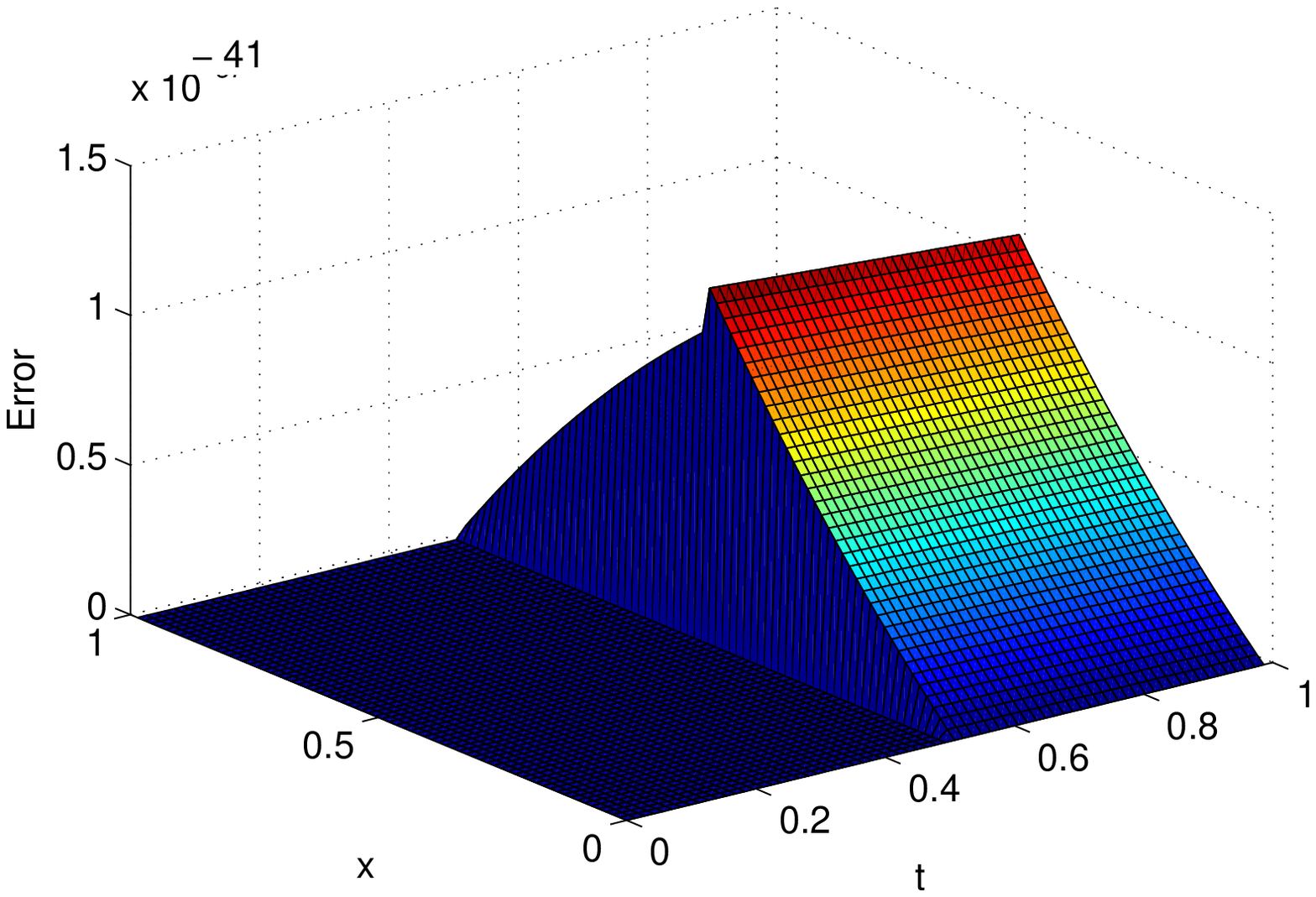,scale=0.48,clip=}
\end{tabular}
 \caption{\small The numerical behavior of the   approximate   solution (left) and absolute error (right) for Example \ref{ex1}.  }
 \label{Fg1}
\end{figure}

\begin{figure}[h!]
\begin{tabular}{ll}
\epsfig{file=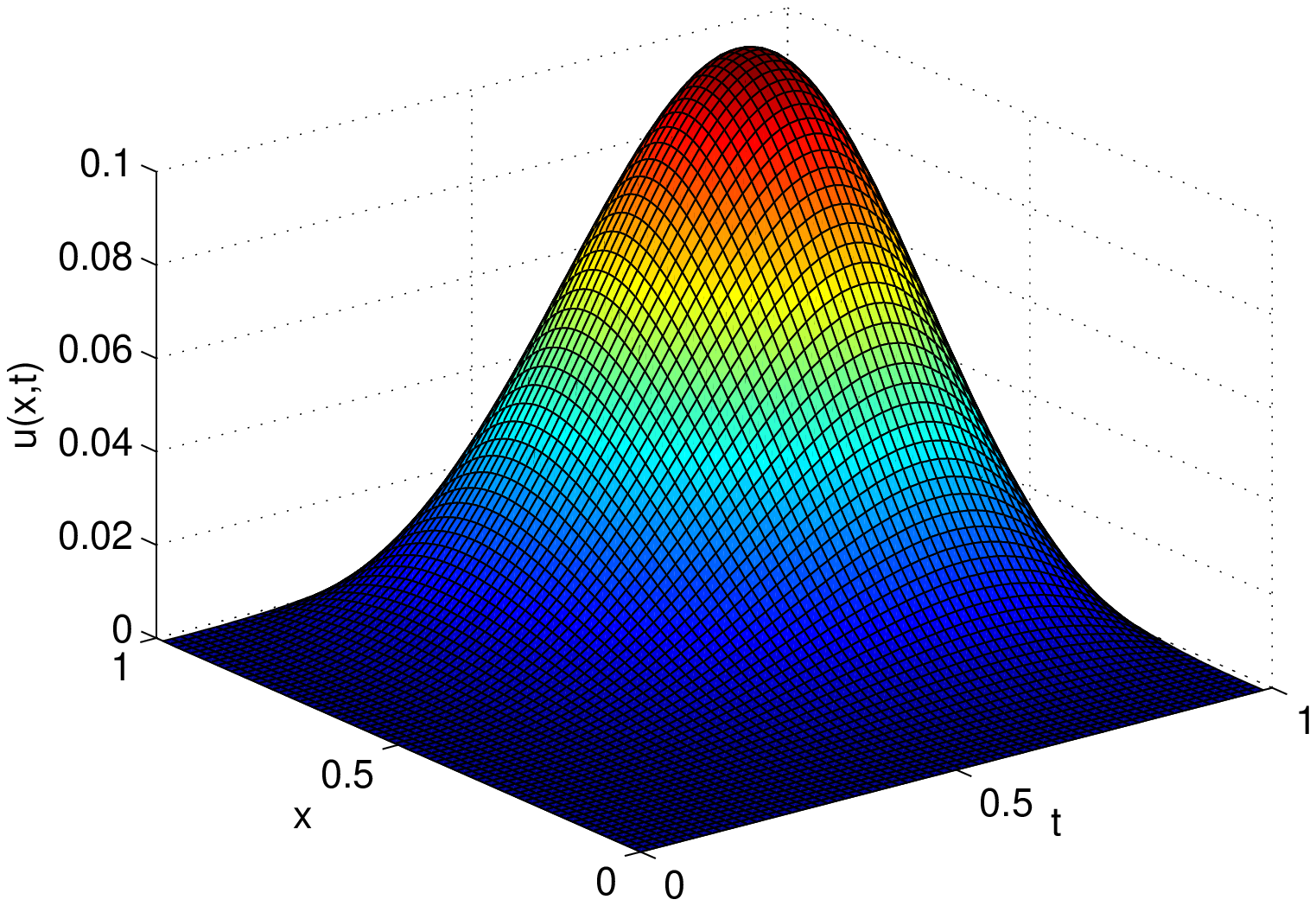,scale=0.48,clip=} &
\epsfig{file=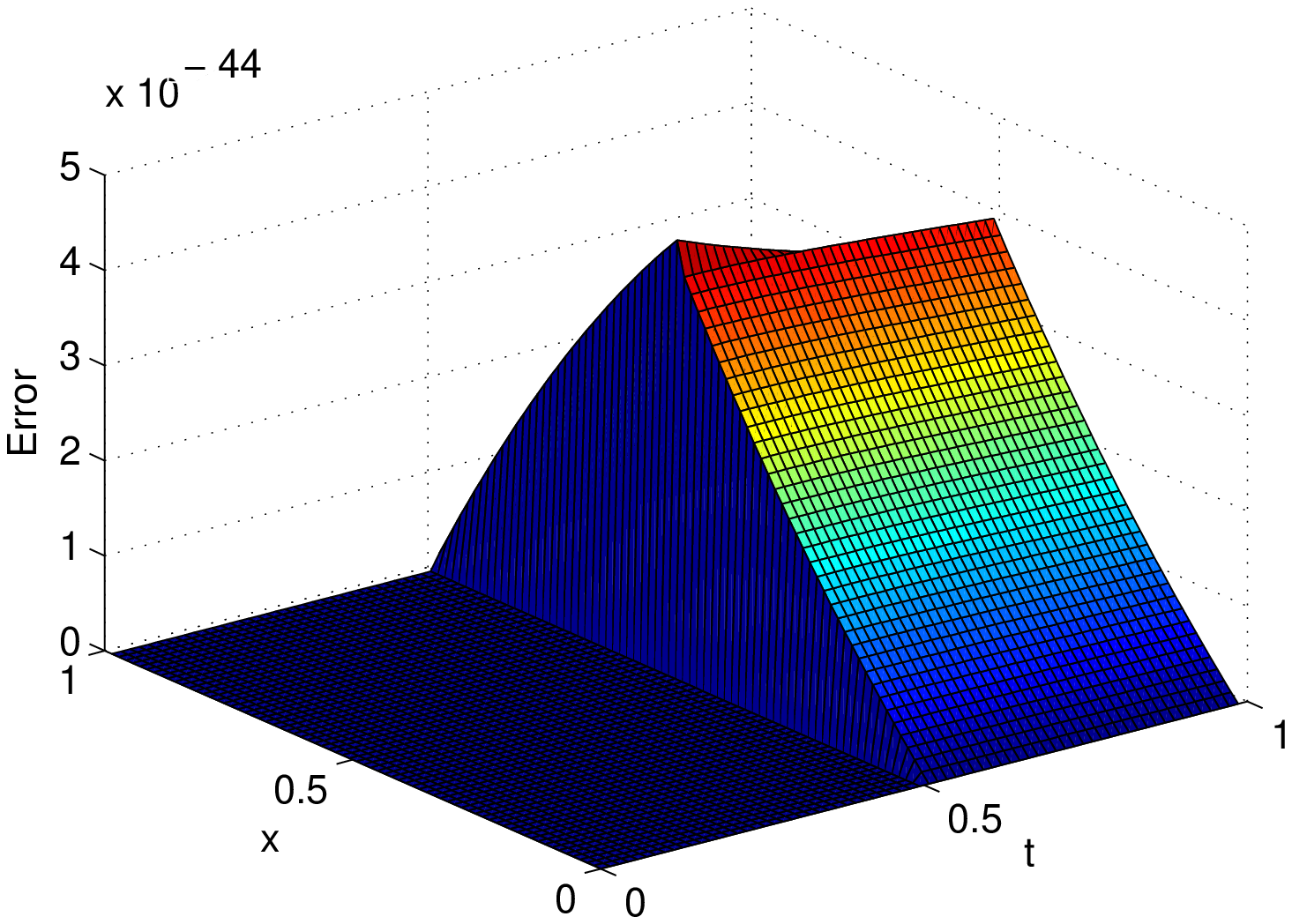,scale=0.48,clip=}
\end{tabular}
 \caption{\small The numerical behavior of the   approximate   solution (left) and absolute error (right) for Example \ref{ex2}.  }
 \label{Fg2}
\end{figure}

\begin{figure}[h!]
\begin{tabular}{ll}
\epsfig{file=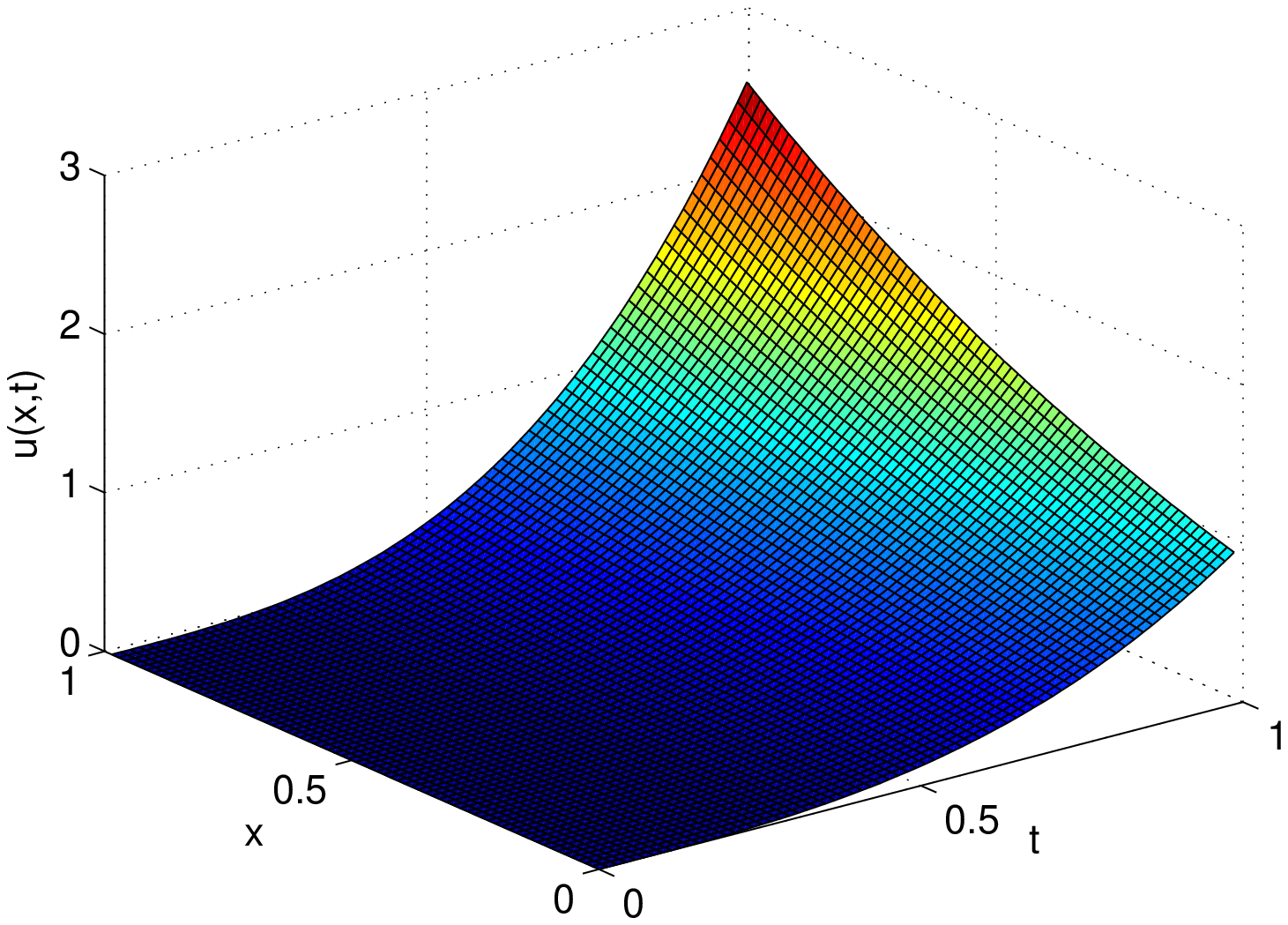,scale=0.48,clip=} &
\epsfig{file=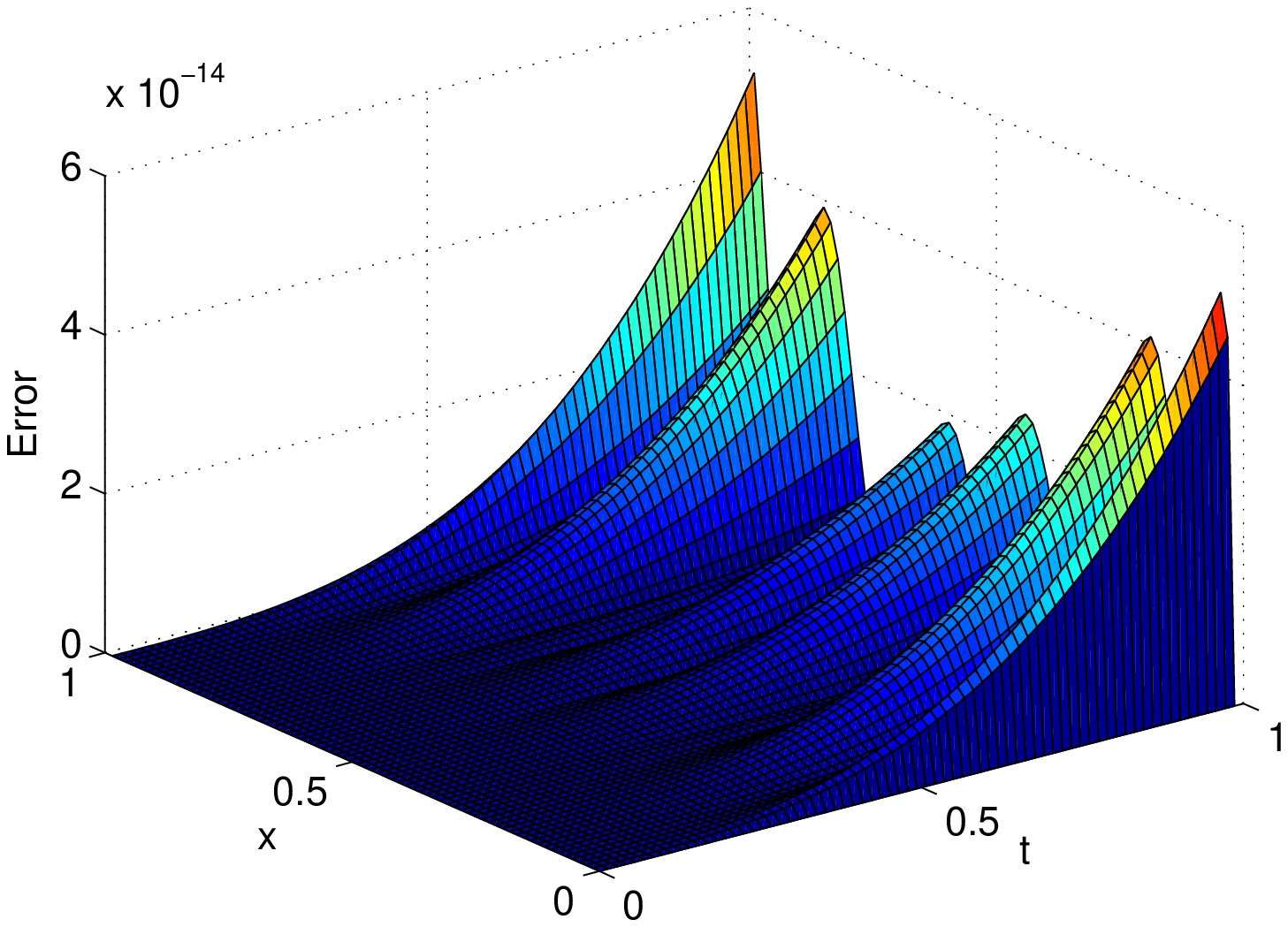,scale=0.48,clip=}
\end{tabular}
 \caption{\small The numerical behavior of the   approximate   solution (left) and absolute error (right) for Example \ref{ex3}.  }
 \label{Fg3}
\end{figure}

\begin{figure}[h!]
\begin{tabular}{ll}
\epsfig{file=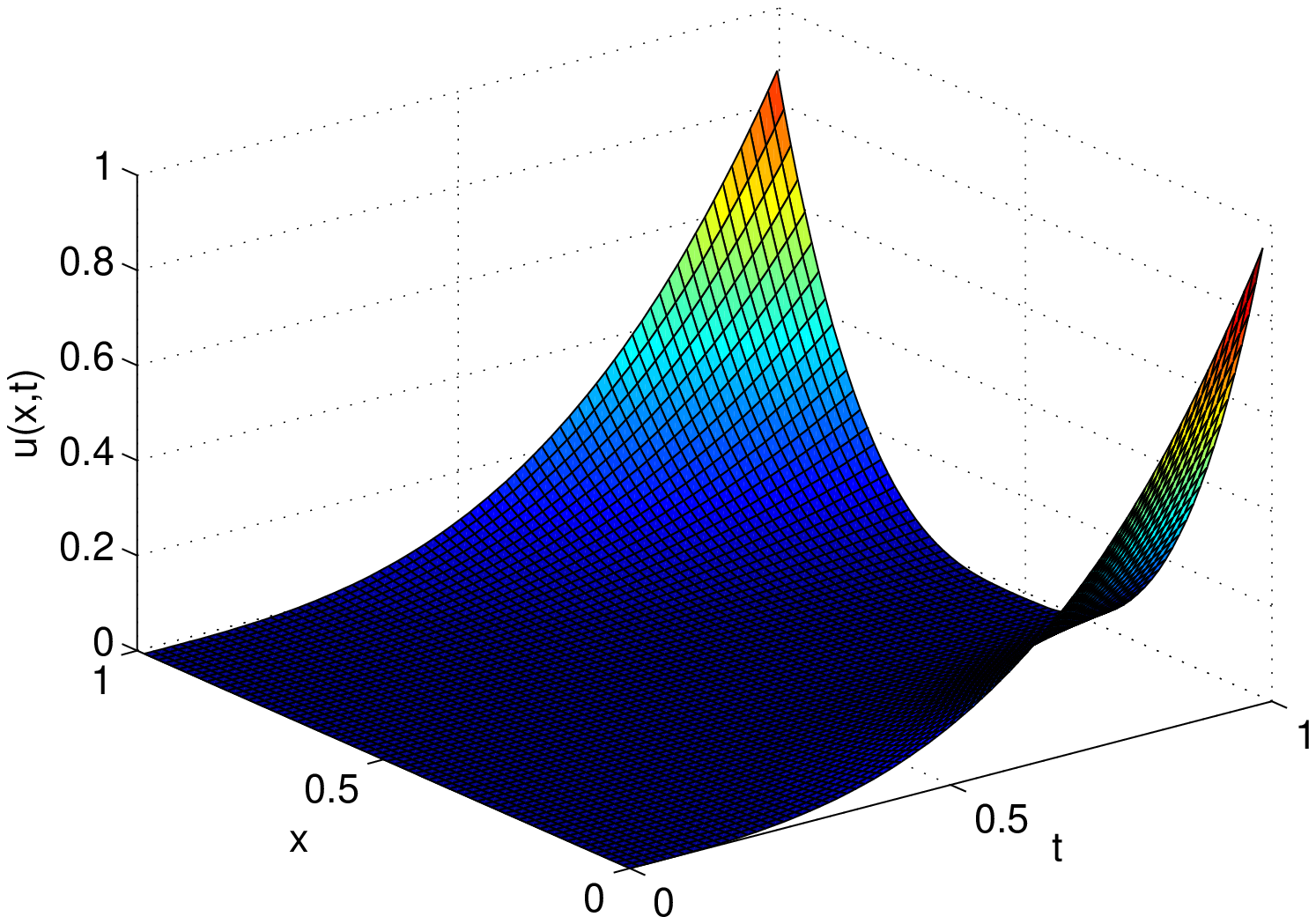,scale=0.48,clip=} &
\epsfig{file=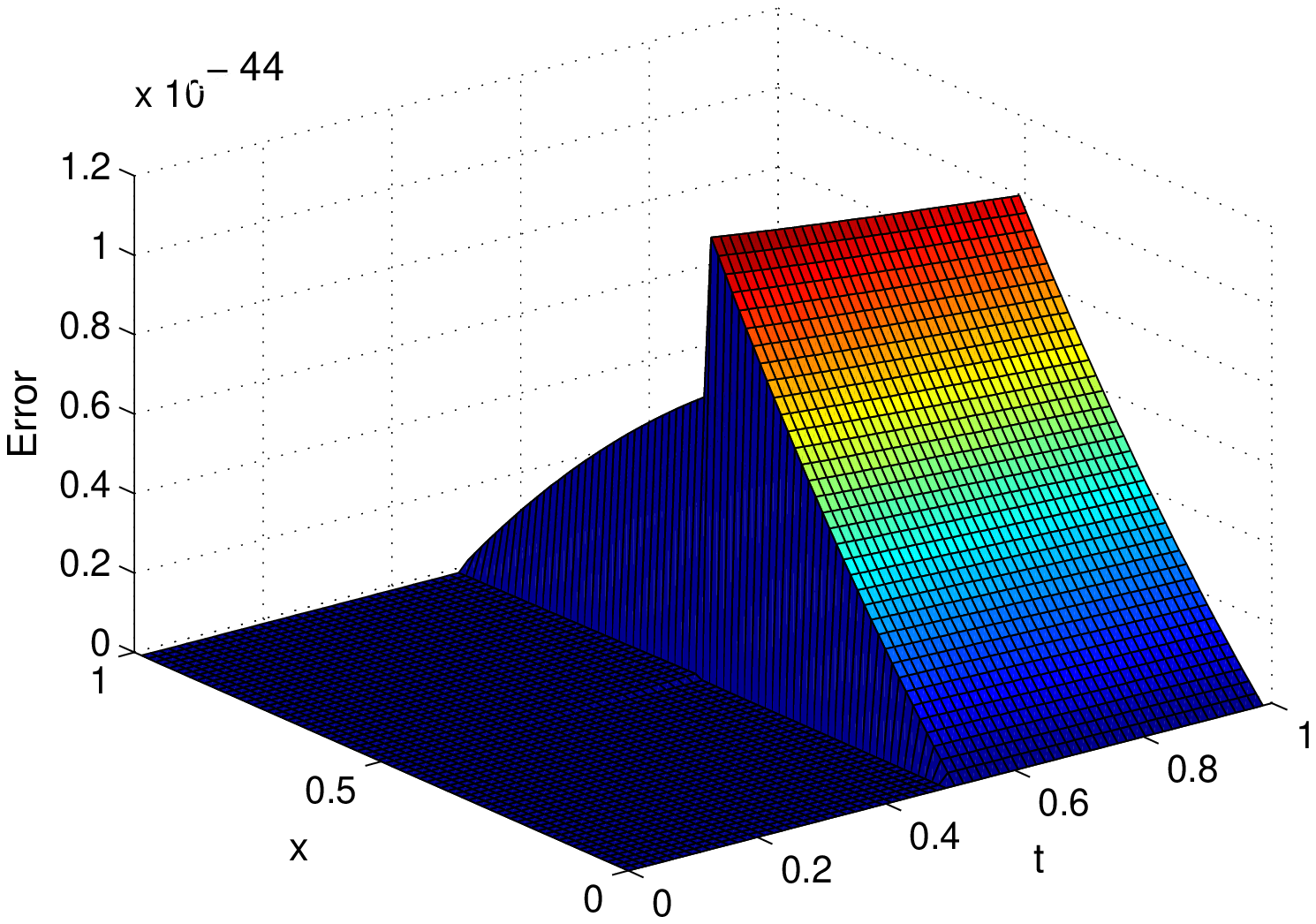,scale=0.48,clip=}
\end{tabular}
 \caption{\small The numerical behavior of the   approximate   solution (left) and absolute error (right) for Example \ref{ex4}.  }
 \label{Fg4}
\end{figure}
\section{Conclusion}\label{S6}
In this paper, a new  numerical method based on the  CWs was proposed
to obtain an approximate solution for the variable-order time  fractional mobile-immobile advection-dispersion model. To this end, a new  operational matrix of variable-order fractional derivative for the  CWs was obtained and  employed  to obtain the approximate solution for the problem under study. Along the way a new family of piecewise functions was introduced and used to obtain a general approach for forming this matrix.
 In the proposed method,  solution of the problem under consideration was expanded in terms of the the  CWs.
The operational matrix of variable-order fractional derivative  and some properties of CWs were employed to transform
 its solution to the solution of a linear system of algebraic equations, which greatly simplified the problem as well as  achieved a good approximate solution for it.  Our proposed method is very efficient and convenient in solving such initial-boundary value problems because all the  conditions are used.  Also, the implementation of the proposed method is very simple for   solution of  the  problem under consideration.
 The  accuracy  of the proposed method was shown for some examples, which shows that our proposed method is  very  accurate for the problem under study.


\end{document}